\documentclass[12pt]{article}
\usepackage[a4paper,margin=0.9in]{geometry}

\usepackage[colorlinks,citecolor=magenta,linkcolor=black]{hyperref}
\pdfpagewidth=\paperwidth \pdfpageheight=\paperheight
\usepackage{amsfonts,amssymb,amsthm,amsmath,eucal,tabu,url}
 \usepackage[noend]{algpseudocode}
 \usepackage[cp1250]{inputenc}
 \usepackage{array}
 \usepackage{pstricks}
 \usepackage{pstricks-add}
 \usepackage{pgf,tikz}
 \usepackage{float}
 \usetikzlibrary{automata}
 \usetikzlibrary{arrows}
 \usepackage{indentfirst}
 \usepackage[all,cmtip]{xy}
 \usepackage[style=empty]{clrdblpg}
\usepackage{tabularx} 

\theoremstyle{plain}
\newtheorem{thm}{Theorem}[section]
\newtheorem{theorem}[thm]{Theorem}
\newtheorem*{theoremA}{Main Theorem}

\newtheorem{proposition}[thm]{Proposition}
\newtheorem{corollary}[thm]{Corollary}

\theoremstyle{definition}
\newtheorem{definition}[thm]{Definition}
\newtheorem{remark}[thm]{Remark}
\newtheorem{example}[thm]{Example}

\newtheorem{thevarthm}[thm]{\varthmname}

\newenvironment{varthm*}[1]{\trivlist\item[]{\bf #1.}\it}{\endtrivlist}

\renewcommand\geq{\geqslant}

\renewcommand\leq{\leqslant}

\newcommand\be{\begin{eqnarray*}}
\newcommand\ee{\end{eqnarray*}}

\renewcommand\P{\mathbb P}

\newcommand\newop[2]{\def#1{\mathop{\rm #2}\nolimits}}
\newop\log{log}
\newop\ord{ord}
\newop\Gal{Gal}
\newop\SL{SL}
\newop\Bl{Bl}
\newop\mult{mult}
\newop\mass{mass}
\newop\div{div}
\newop\codim{codim}
\newop\sing{sing}
\newop\vdim{vdim}
\newop\edim{edim}
\newop\Ass{Ass}
\newop\size{size}
\newop\reg{reg}
\newop\satdeg{satdeg}
\newop\supp{supp}
\newop\Neg{Neg}
\newop\Nef{Nef}
\newop\Nefh{Nef_H}
\newop\Eff{Eff}
\newop\Zar{Zar}
\newop\MB{MB}
\newop\MBxC{MB\mathit{(x,C)}}
\newop\NnB{NnB}
\newop\Bigg{Big}
\newop\Effbar{\overline{\Eff}}

\def\keywordname{{\bfseries Keywords}}%
\def\keywords#1{\par\addvspace\medskipamount{\rightskip=0pt plus1cm
\def\and{\ifhmode\unskip\nobreak\fi\ $\cdot$
}\noindent\keywordname\enspace\ignorespaces#1\par}}
\def\subclassname{{\bfseries Mathematics Subject Classification
(2020)}\enspace}
\def\subclass#1{\par\addvspace\medskipamount{\rightskip=0pt plus1cm
\def\and{\ifhmode\unskip\nobreak\fi\ $\cdot$
}\noindent\subclassname\ignorespaces#1\par}}

\begin{document}
\title{On the containment $I^{(3)} \subset I^{2}$ and configurations of triple points in B\"or\"oczky line arrangements}
\author{Jakub Kabat}
\date{\today}
\maketitle

\thispagestyle{empty}
\begin{abstract}
We study sets of triple points of B\"or\"oczky's arrangements of lines in the context of the containment problem proposed by Harbourne and Huneke. We show that in the class of those arrangements, the smallest counterexample to the containment $I^{(3)} \subset I^{2}$ is obtained when the number of lines is equal to $12$.
\subclass{14C20, 13C05, 52C35}
\end{abstract}

\section{Introduction}
\label{chap:containment}
In the present paper we study the so-called containment problem for fat point schemes in the projective plane. Let $\mathcal{P} = \{P_{1}, ..., P_{s} \} \subset \mathbb{P}^{2}_{\mathbb{C}}$ be a finite set of mutually distinct points. We denote by $I(\mathcal{P}) := I$ the associated ideal, i.e.,
$$I(\mathcal{P}) = I(P_{1}) \cap  ... \cap I(P_{s}),$$
where $I(P_{i})$ is the defining ideal of point $P_{i}$. Then, for every $m\geq 1$, we define the $m$-th symbolic power of $I$, denoted here by $I^{(m)}$, as

$$I^{(m)} = I^{m}(P_{1}) \cap ... \cap I^{m}(P_{s}).$$
Using the celebrated result by Nagata and Zariski we know that the $m$-symbolic power of $I$ consists of all homogeneous forms in $\mathbb{C}[x,y,z]$ vanishing along $V(I)$ with multiplicity at least $m$. It is somehow natural to ask whether there is a relation between symbolic and ordinary powers of ideals. Here is the chronological list of problems and achievements regarding the mentioned relations. 
\begin{enumerate}
	\item[(\textbf{2001})] Ein, Lazarsfeld, and Smith \cite{ELS01}: $I^{(2k)} \subset I^{k}$ for every $k\geq 1$.
	\item[(\textbf{2006})] Huneke: Does the containment $I^{(3)} \subset I^{2}$ hold?
	\item[(\textbf{2009})] Bocci and Harbourne: Does the containment $I^{(2k-1)} \subset I^{k}$ hold for every $k\geq 1$?
	\item[(\textbf{2013})] Dumnicki, Szemberg, and Tutaj-Gasi\'nska \cite{DST13}: The first counterexample to the containment $I^{(3)} \subset I^{2}$ -- they used the dual-Hesse arrangement of $9$ lines and $12$ triple intersection points.
	\item[(\textbf{2013})] Czapli\'nski \emph{et al.} \cite{Real}: The first counterexample to the containment $I^{(3)} \subset I^{2}$ over the real numbers -- B\"or\"oczky's arrangement of $12$ lines, $19$ triple and $9$ double intersection points.
	\item[(\textbf{2015})] Lampa-Baczy\'nska and Szpond \cite{LBSzp16}: The first counterexample to the containment $I^{(3)} \subset I^{2}$ over the rational numbers -- using the parameter space of B\"or\"oczky arrangement of $12$ lines they found a rational realization of this combinatorics.
	\item[(\textbf{2015})] Harbourne: Construct new counterexamples to the containment  $I^{(3)} \subset I^{2}$ over the rational numbers using parameter spaces of B\"or\"oczky's line arrangements.
\end{enumerate}
The main aim of the present note is to investigate the radical ideals of the triple intersection points in B\"or\"oczky's arrangements of $n \in \{4,...,11\}$ in order to verify whether the minimal counterexample to Huneke's question in this class of arrangements is obtained for $n=12$ lines. Here by the minimal counterexample we mean the minimal number of lines. The main result of the present note can be formulated as follows.
\begin{theoremA}
Let $\mathcal{P}_{n} \subset \mathbb{P}^{2}_{\mathbb{C}}$ be a set of triple intersection points of the B\"or\"oczky's arrangement of $n$ lines. Denote by $I_{n}$ the associated radical ideal of $\mathcal{P}_{n}$. Then 
$$I^{(3)}_{n} \subset I_{n}^{2}$$ holds provided that $n \in \{4, ...,11\}$. In other words, $n=12$ is the minimal number of lines for which the containment
$$I^{(3)}_{n} \subset I^{2}_{n}$$ does not hold.
\end{theoremA}

All necessary details regarding B\"or\"oczky's line arrangements will be delivered in the forthcoming section. In Section 3, we present necessary tools that we need to prove our results, and in Section 4 we present our proof on Main Theorem (which is divided into two separate results).
\section{B\"or\"oczky's line arrangements}

In this subsection, we describe the main construction, namely B\"or\"oczky's arrangements $\mathcal{B}_{n}$ which were introduced in \cite[Example 2]{FuPa1984}. Following this example, we present here an outline of the construction.

   Consider a regular $2n$-gon inscribed in the unit circle in the real affine plane.
   Let us fix one of the $2n$ vertices and denote it by $Q_0$.
   By $Q_{\alpha}$ we denote the point arising by the rotation of $Q_0$ around the center of the circle by angle $\alpha$.

   Then we take the following set of lines
$$\mathcal{B}_{n} =\left\{Q_{\alpha}Q_{\pi - 2\alpha}, \textrm{where } \alpha= \frac{2k \pi}{n} \textrm{ for } k=0, \dots, n-1\right\}.$$
   If $\alpha \equiv (\pi - 2\alpha)({\rm mod}\; 2\pi)$,
   then the line $Q_{\alpha}Q_{\pi - 2\alpha}$ is the tangent to the circle at the point $Q_{\alpha}$.
   The arrangement $\mathcal{B}_{n}$ has $\big\lfloor \frac{n(n-3)}{6}\big\rfloor+1$ triple points by \cite[Property 4]{FuPa1984}, and we denote this set of triple points by $\mathbb{T}_n$. On Figure \ref{n=122} we depicted B\"or\"oczky's arrangement of $n=12$ lines. 
\begin{example}[B\"or\"oczky arrangement of $12$ lines]\label{ex:B12}
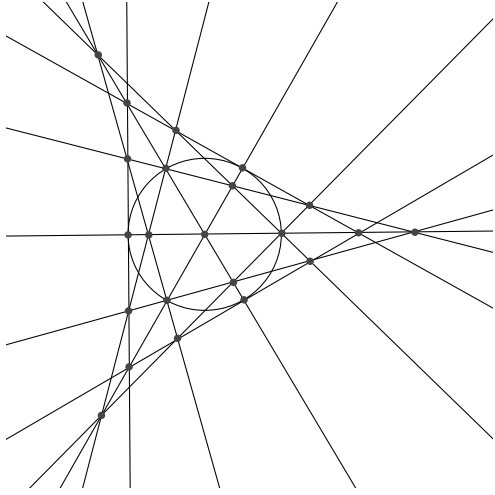
\begin{figure}[ht]
\centering
\definecolor{uuuuuu}{rgb}{0.27,0.27,0.27}
\begin{tikzpicture}[line cap=round,line join=round,x=1.0cm,y=1.0cm,scale = 0.7]
\clip(-0.82,-3.96) rectangle (8.3,5.22);
\draw [domain=-4.36:18.32] plot(\x,{(-2.33-0.03*\x)/-2.89});
\draw [domain=-4.36:18.32] plot(\x,{(-5.21--0.51*\x)/-1.98});
\draw [domain=-4.36:18.32] plot(\x,{(--7.45-1.43*\x)/1.46});
\draw [domain=-4.36:18.32] plot(\x,{(--8.44-2.49*\x)/1.47});
\draw [domain=-4.36:18.32] plot(\x,{(--4.08-1.97*\x)/0.55});
\draw [domain=-4.36:18.32] plot(\x,{(-3.23--1.98*\x)/0.51});
\draw [domain=-4.36:18.32] plot(\x,{(-6.11--2.51*\x)/1.42});
\draw [domain=-4.36:18.32] plot(\x,{(-5.12--1.46*\x)/1.43});
\draw [domain=-4.36:18.32] plot(\x,{(--2.03-0.55*\x)/-1.97});
\draw [domain=-4.36:18.32] plot(\x,{(--5.19-0.71*\x)/1.26});
\draw [domain=-4.36:18.32] plot(\x,{(--3.18-0.73*\x)/-1.24});
\draw [domain=1:2] plot(\x,{(-2.12--1.44*\x)/-0.01});
\draw(2.9,0.83) circle (1.44cm);
\begin{scriptsize}
\fill [color=uuuuuu] (2.19,-0.42) circle (2.0pt);
\fill [color=uuuuuu] (4.35,0.85) circle (2.0pt);
\fill [color=uuuuuu] (2.17,2.08) circle (2.0pt);
\fill [color=uuuuuu] (0.9,4.23) circle (2.0pt);
\fill [color=uuuuuu] (1.44,3.32) circle (2.0pt);
\fill [color=uuuuuu] (1.45,2.26) circle (2.0pt);
\fill [color=uuuuuu] (2.36,2.8) circle (2.0pt);
\fill [color=uuuuuu] (1.85,0.82) circle (2.0pt);
\fill [color=uuuuuu] (2.9,0.83) circle (2.0pt);
\fill [color=uuuuuu] (1.48,-1.68) circle (2.0pt);
\fill [color=uuuuuu] (0.96,-2.6) circle (2.0pt);
\fill [color=uuuuuu] (2.39,-1.14) circle (2.0pt);
\fill [color=uuuuuu] (1.47,-0.62) circle (2.0pt);
\fill [color=uuuuuu] (3.44,-0.08) circle (2.0pt);
\fill [color=uuuuuu] (4.88,0.32) circle (2.0pt);
\fill [color=uuuuuu] (4.87,1.38) circle (2.0pt);
\fill [color=uuuuuu] (5.79,0.86) circle (2.0pt);
\fill [color=uuuuuu] (6.85,0.87) circle (2.0pt);
\fill [color=uuuuuu] (3.42,1.75) circle (2.0pt);
\fill [color=uuuuuu] (3.61,2.09) circle (2.0pt);
\fill [color=uuuuuu] (3.64,-0.41) circle (2.0pt);
\fill [color=uuuuuu] (1.46,0.82) circle (2.0pt);
\end{scriptsize}
\end{tikzpicture}
\caption{B\"or\"oczky arrangement of $12$ lines.}
\label{n=122}
\end{figure}
\end{example}
   In the sequel we will need the following simple fact concerning the distribution of triple points on the arrangement lines. These results have been already presented in \cite[Proposition 3.3 and Corollary 3.4]{Kabat}, but for the completeness we recall them there.
   
\begin{proposition}\label{prop:triple point on B lines}
   Every line in the $\mathcal{B}_n$ arrangement contains at least $\big\lfloor\frac{n-3}{2}\big\rfloor$ triple points
   and there exists a line containing at least one more triple point.
\end{proposition}

\begin{proposition}\label{cor:alpha for Bn}
   For a fixed $n\geq 8$ let $C$ be a plane curve (possibly reducible and non-reduced) of degree $d$ passing through every point
   in the set $\mathbb{T}_n$ with multiplicity at least $3$. Then $d\geq n$. Moreover, if $d=n$, then $C$ is the union of
   all arrangement lines in $\mathcal{B}_n$.
\end{proposition}

\section{Containment criteria}
   Over the years a number of containment criteria has been developed. We recall here those which are relevant for our applications. We begin by recalling some standard notions in a general setting of homogeneous ideals in the ring of polynomials. In order to fix the notation, let
   $I$ be a homogeneous ideal in the polynomial ring $R=\mathbb{K}[x_0,\ldots,x_n]$. Let
   $$0\to\ldots\to\oplus_jR(-j)^{\beta_{i,j}(I)}\to\ldots \to\oplus_jR(-j)^{\beta_{1,j}(I)}\to\oplus_jR(-j)^{\beta_{0,j}(I)}\to I\to 0$$
   be the minimal free resolution of $I$. From this resolution
   we derive one of central invariants in commutative algebra and algebraic geometry.
\begin{definition}\label{def:CM}
   The Castelnuovo-Mumford regularity (or simply, regularity) of $I$, denoted by $\reg(I)$,
   is the integer
   $$\reg(I)=\max\left\{j-i:\; \beta_{i,j}(I)\neq 0\right\}.$$
\end{definition}
   Thus $\reg(I)$ is the height of the Betti table of $I$.

   Another important invariant of a homogeneous ideal $I=\oplus_{t=0}^{\infty}(I)_t$ is its initial degree
   $$\alpha(I)=\min\left\{t:\; (I)_t\neq 0\right\}=\min\left\{j:\;\beta_{0,j}\neq 0\right\}.$$
   Note that it is always
   $$\alpha(I)\leq \reg(I)$$
   because $\reg(I)$ is at least equal to the maximal degree of a generator in the minimal set of generators.

   Bocci and Harbourne proved in \cite[Lemma 2.3.3 (c)]{BocHar10a} an important containment statement, which we recall here only in the case of saturated ideals of zero-dimensional subschemes in $\P^n$.
\begin{proposition}[Bocci-Harbourne Containment Criterion]\label{cri:BH}
   Let $I\subset R$ be a non-trivial saturated homogeneous ideal defining a zero-dimensional subscheme. For
   $t\geq r\cdot\reg(I)$ there is
   $$(I^r)_t=(I^{(r)})_t.$$
\end{proposition}
\begin{remark}
	\label{rem:bh}
   It follows from the proof of Lemma 2.3.3 in \cite{BocHar10a} that the conclusion in Proposition \ref{cri:BH} holds as soon, as $t\geq \reg(I^r)$.
\end{remark}
   From Proposition \ref{cri:BH} and Remark \ref{rem:bh} we derive the following useful result.
\begin{corollary}[Bocci-Harbourbe Containment Criterion 2]
	\label{cri:bh2}
   Let $I\subset R$ be a non-trivial ideal defining a zero-dimensional subscheme in $\P^n$.
   $$\mbox{If } \reg(I^{r})\leq\alpha(I^{(m)}),
   \mbox{ then } I^{(m)}\subset I^r.$$
\end{corollary}

   In the rest of this section, we consider zero-dimensional strict almost complete intersections, i.e., ideals of height $h$ that have a minimal set of generators of cardinality $h+1$. In the case of projective plane, a reduced set of points is strict almost complete intersection if its ideal is $3$-generated -- the minimal set of homogeneous generators of degree $d$ has cardinality $3$.
	Let $ I = (f,g,h) \subset R:=\mathbb{K}[x,y,z]$ (here we do not assume anything about $\mathbb{K}$) be a homogeneous ideal with minimal generators of the same degree. We are interested in free resolutions for powers of $I$, and in order to do so we need to consider the Rees algebra of $I$, which is defined by $\mathcal{R}(I) = \oplus_{i \geq 0}I^{i}t^{i}$. In that case we have the following description, see \cite{Sece}.
\begin{theorem}
Let $I$ be a strict almost complete intersection ideal defining a reduced set of points in $\mathbb{P}^{2}$ and let $A^{T} =\left( \begin{array}{ccc}
P_{1} & P_{2} & P_{3} \\
Q_{1} & Q_{2} & Q_{3}
 \end{array} \right) $ be a presentation matrix for the module of syzygies on $I$, i.e., the Hilbert-Burch matrix of $I$. Then the Rees algebra of $I$ is given as a quotient of the polynomial ring $S = R(T_{1},T_{2},T_{3})$ of the following form
$$\mathcal{R}(I) \cong S/(P_{1}T_{1} + P_{2}T_{2} + P_{3}T_{3}, Q_{1}T_{1}+Q_{2}T_{2}+Q_{3}T_{3}).$$
Furthermore, the defining ideal of this algebra, $(P_{1}T_{1} + P_{2}T_{2} + P_{3}T_{3}, Q_{1}T_{1} + Q_{2}T_{2} + Q_{3}T_{3})$ is a complete intersection.
\end{theorem}
Before we present our main tool, we need the following result providing a precise description of powers of strict almost complete intersection ideals.
\begin{theorem}
Let $I$ be a strict almost complete intersection ideal with minimal generators of the same degree $d$ defining a reduced set of points in $\mathbb{P}^{2}_{\mathbb{K}}$. Let $A^{T} =\left( \begin{array}{ccc}
P_{1} & P_{2} & P_{3} \\
Q_{1} & Q_{2} & Q_{3}
\end{array} \right)$
be the Hilbert-Burch matrix of $I$. Let $d_{0}$ and $d_{1}$ denote the respective degrees of the polynomials in each of the two rows of $A^{T}$. Then the minimal free resolutions of $I^2$ and $I^3$ are as follows:
$$0 \rightarrow R(-3d) \stackrel{X}{\longrightarrow} R(-2d - d_{0})^{3}\oplus R(-2d-d_{1})^{3} \longrightarrow R(-2d)^{6} \longrightarrow I^{2} \longrightarrow 0,$$
$$ 0 \longrightarrow R(-4d)^{3} \stackrel{Y}{\longrightarrow}R(-3d-d_{0})^{6} \oplus R(-3d - d_{1})^{6} \longrightarrow R(-3d)^{10} \longrightarrow I^{3} \longrightarrow 0,$$
and the last homomorphism in the respective resolutions can be described by the matrices $X$ and $Y$ given below by:
$$X = [P_{1}, \quad P_{2}, \quad P_{3}, \quad -Q_{1}, \quad -Q_{2}, \quad -Q_{3}]^T,$$
and
$$Y =\left( \begin{array}{cccccccccccc}
P_{1} & P_{2} & P_{3} & 0 & 0 & 0 & -Q_{1} & -Q_{2} & -Q_{3} & 0 & 0 & 0 \\
0 & P_{1} & 0 & P_{2} & P_{3} & 0 & 0 & -Q_{1} & 0 & -Q_{2} & -Q_{3} & 0 \\
0 & 0 & P_{1} & 0 & P_{2} & P_{3} & 0 & 0 & -Q_{1} & 0 & -Q_{2} & -Q_{3}
\end{array} \right)^T .$$
\end{theorem}
\begin{theorem}[Seceleanu]
	Let $I$ be a $3$-generated homogeneous ideal with minimal generators $f,g,h$ of the same degree $d$, defining a reduced set of points in $\mathbb{P}^{2}_{\mathbb{K}}$, where $\mathbb{K}$ is an arbitrary field of characteristic different than $3$. Set $Y$ to be the matrix representing the last homomorphism in the minimal free resolution of $I^{3}$ (see above):
	$$0 \longrightarrow R^{3} \stackrel{Y}{\longrightarrow} R^{12} \longrightarrow R^{10} \longrightarrow I^{3} \longrightarrow 0.$$
	Then $I^{(3)} \subseteq I^{2}$ if and only if $[f,g,h]^{T} \in {\rm Image}(Y^{T})$.
\end{theorem}
In the same direction, we can follow ideas of Grifo, Huneke, and Mukundan developed in \cite{Grifo}. In order to formulate more efficient criterion on the containment $I^{3} \subset I^2$ for ideals generated by $2 \times 2$ minors of $2 \times 3$ matrices.
\begin{theorem}[Grifo-Huneke-Mukundan]
	Let $R = \mathbb{K}[x,y,z]$, where $\mathbb{K}$ is a field of characteristic different than $3$. Let $a_{1}, a_{2}, a_{3}, b_{1}, b_{2}, b_{3} \in R$ and consider the ideal $I$ which is generated by $2 \times 2$ minors of the matrix
	$$A = \left( \begin{array}{ccc}
	a_{1} & a_{2} & a_{3} \\
	b_{1} & b_{2} & b_{3}
	\end{array} \right).$$
	If the ideal $\langle a_{1},a_{2},a_{3},b_{1},b_{2},b_{3}\rangle$ can be generated by $5$ or less elements, then $I^{(3)} \subset I^{2}$.
\end{theorem}
\section{Containment results for some B\"or\"oczky's line arrangements}
It turns out that we can use this interesting result in a straightforward way in the case of B\"or\"oczky's arrangements of $n \in \{4, ..., 10\}$ lines in order to verify that for the radical ideals of triple intersection points $I_{3}$ the containment $I^{(3)} \subset I^{2}$ \textbf{does} hold.  Since the method is the same for all cases, we are going to present our considerations only for $n=10$.

\begin{proposition}
Let $I_{3}$ be the radical ideal of the triple intersection points of B\"or\"oczky's arrangement of $10$ lines. Then the containment $I^{(3)}_{3} \subset I_{3}^{2}$ does hold.
\end{proposition}
\begin{proof}
First of all, we need to observe that the ideal of the triple intersection points is generated as bellow by
$$I_{3} = \langle 4xy^{3} + 2x^{2}yz + 4y^{3}z - xyz^{2} - 3yz^{3}, 4x^{3}y + 2x^{2}yz - 3xyz^{2} - yz^{3}, $$ $$ x^{4} - 6x^{2}y^{2} + y^{4} - 4x^{3}z + x^{2}z^{2} + y^{2}z^{2} + 2xz^{3} - z^{4} \rangle.$$
Since the ideal $I_{3}$ is $3$-generated, we can use the theory of Hilbert-Burch. We compute the minimal free resolution of $I_{3}$, and the matrix $A$ that we are searching for is given by the following Hilbert-Burch matrix, namely
$$A = \left( \begin{array}{ccc}
4x^{2} - 2xz -z^{2} & 4y^{2} - 14xz + z^{2} & -4y^{2} - 2xz + 3z^{2} \\
4x^{2} - 24y^{2} -14 xz + 13z^{2} & 0 & -16xy - 16 yz
\end{array} \right).$$
Since it is obvious that the ideal given by the entries of matrix $A$ is $5$ or less generated (in fact it is $5$ generated), thus the containment $I^{(3)}_{3} \subset I^{2}_{3}$ holds.
\end{proof}
Now we are going to consider the last remaining case which would allow us to conclude that the minimal counterexample to the containment problem $I^{(3)} \subset I^{2}$ (in the sense of the number of lines) for B\"or\"oczky's family of line arrangements is the case of $12$ lines. As a first observation, we can show that for $n=11$ lines the ideal of the triple intersection points is not $3$-generated -- in fact the minimal set of generators has cardinality $4$, so we cannot use the Grifo-Huneke-Mukundan method. In the remaining part of this section, we are going to show explicitly the following theorem.
\begin{theorem}
Let us denote by $I_{3}$ the radical ideal of the triple intersection points of B\"or\"oczky's arrangement of $11$ lines. Then the containment $I^{(3)}_{3} \subset I_{3}^{2}$ holds.	
\end{theorem}
\begin{proof}
Our proof heavily relies on computer aid methods with use of Singular. First of all, we compute the ideal $I_{3}$ which has exactly $4$ generators, namely
\begin{align*}
I_{3} =& \langle 4x^3y-4xy^3-3x^2yz-3y^3z-2xyz^2+2yz^3,\\
&32y^5+88xy^3z+33x^2yz^2-55y^3z^2-66xyz^3+22yz^4,\\
&32x^2y^3+72xy^3z+11x^2yz^2+35y^3z^2-22xyz^3-22yz^4,\\
&2x^5-10xy^4-8x^4z-15x^2y^2z-7y^4z+4x^3z^2+2xy^2z^2+10x^2z^3+8y^2z^3-4xz^4-2z^5\rangle.
\end{align*}
Then we compute the minimal free resolution of $I^{2}_{3}$ which has the following form
\begin{align*}
0 \rightarrow S(-13)^{2}\oplus S(-12) \rightarrow S(-12)^{3} \oplus S(-11)^{7} \oplus S(-10)^{2} \\ \rightarrow S(-10)^{6}\oplus S(-9)^{3}\oplus S(-8) \rightarrow I^{2}_{3} \rightarrow 0.
\end{align*}

Thus we have ${\rm reg}(I^{2}_{3}) =11$. Taking into account Corollary \ref{cor:alpha for Bn} we obtain $\alpha(I_{3}^{(3)}) = 11$. Applying in turn Corollary \ref{cri:bh2} with $m=3$ and $r=2$ we conclude that
$$I^{(3)}_{3} \subset I_{3}^{2}.$$
\end{proof}
\section*{Acknowledgments}
The author was partially supported by the National Science Center (Poland) Preludium Grant \textbf{Nr UMO 2018/31/N/ST1/02101}. 

\vskip 0.5 cm

\bigskip
Jakub Kabat, \\
Department of Mathematics,
Pedagogical University of Krakow,
ul. Podchorazych 2,
PL-30-084 Krak\'ow, Poland. \\
\nopagebreak
\textit{E-mail address:} \texttt{jakub.kabat@up.krakow.pl}
\bigskip

\begin{thebibliography}{000}
\bibitem{BocHar10a}
Bocci, C., Harbourne, B.:
Comparing Powers and Symbolic Powers of Ideals.
\textit{J. Algebraic Geometry} \textbf{19}: 399 -- 417 (2010).
\bibitem{Real}
Czapli\'nski, A., G\l \'owka, A., Malara, G., Lampa-Baczynska, M., \L uszcz-\'Swidecka, P., Pokora, P. and Szpond, J.:
A counterexample to the containment $I^{(3)}\subset I^2$ over the reals. \textit{Adv. Geom.} \textbf{16}: 77 -- 82 (2016).

\bibitem{DGPS}
Decker, W., Greuel, G.-M., Pfister, G., Sch{\"o}nemann, H.:
{\sc Singular} {3-1-3} --- {A} computer algebra system for polynomial computations.
{http://www.singular.uni-kl.de} (2011).

\bibitem{DST13}
Dumnicki, M., Szemberg, T., Tutaj-Gasi\'nska, H.:
Counterexamples to the $I^{(3)}\subset I^2$ containment.
\textit{J.~Alg.} \textbf{393}: 24 -- 29 (2013).

\bibitem{ELS01}
Ein, L., Lazarsfeld, R., Smith, K.:
Uniform bounds and symbolic powers on smooth varieties.
\textit{Invent. Math.} \textbf{144}: 241 -- 252 (2001).
\bibitem{FuPa1984}
F{\"u}redi, Z., Pal\'asti, I.:
Arrangements of lines with a large number of triangles.
\textit{Proc. Amer. Math. Soc.} \textbf{92(4)}: 561 -- 566 (1984).
\bibitem{Grifo}
 Grifo, E., Huneke, C., Mukundan, V.: Expected resurgences and symbolic powers of ideals. \textit{J. London Math. Soc. (2)} \textbf{102}: 453 -- 469 (2020).
 
\bibitem{HH}
Hochster, M., Huneke, C.:
 Comparison of symbolic and ordinary powers of ideals. \textit{Invent. Math.} \textbf{147}: 349 -- 369 (2002).
 \bibitem{Kabat}
Kabat, J.: Supersolvable resolutions of line arrangements. \textbf{arXiv:2201.04856}.
\bibitem{LBSzp16}
Lampa-Baczy\'nska, M., Szpond, J.:
From Pappus Theorem to parameter spaces of some extremal line point configurations and applications. \textit{Geom. Dedicata} \textbf{188}: 103 -- 121 (2017).

\bibitem{Sece}
Seceleanu, A.: A homological criterion for the containment between symbolic and ordinary powers of some ideals of points in $\mathbb{P}^{2}$. \textit{J. Pure Appl. Algebra} \textbf{219(11)}: 4857 -- 4871 (2015).
\end{thebibliography}
\end{document}